\documentclass[12pt,reqno]{amsart}
\usepackage[latin1]{inputenc}
\usepackage[english]{babel}
\usepackage{fourier}
\usepackage[margin=1in]{geometry}
\usepackage{comment}
\usepackage[colorlinks,linkcolor=magenta,citecolor=cyan]{hyperref}
\usepackage{amssymb}
\usepackage{amsfonts}
\usepackage{amsmath}
\usepackage{amsthm}
\usepackage{color}
\usepackage{graphicx}
\usepackage{epsfig,mathrsfs}
\usepackage[bf,SL,BF]{subfigure}
\usepackage{fancyhdr}
\usepackage{graphicx}
\usepackage{subfigure}
\usepackage{caption}
\usepackage{wrapfig}
\usepackage{cases}
\usepackage{makecell}
\usepackage{multirow}
\usepackage{enumitem}
\usepackage{comment}

\allowdisplaybreaks

\newtheorem{theorem}{Theorem}
\newtheorem{lemma}[theorem]{Lemma}

\newtheorem{definition}[theorem]{Definition}

\newtheorem{remark}[theorem]{Remark}

\newtheorem{proposition}[theorem]{Proposition}

\numberwithin{equation}{section}

\numberwithin{theorem}{section}

\numberwithin{subsection}{section}

\newcommand{\R}{\mathbb{R}}
\newcommand{\pa}{\partial}
\newcommand{\LO}[1]{L^{#1}(\Omega)}

\newcommand{\bra}[1]{\left(#1\right)}
\newcommand{\sbra}[1]{\left[#1\right]}
\newcommand{\sumi}{\sum_{i=1}^{m}}
\newcommand{\intO}{\int_{\Omega}}
\newcommand{\eps}{\varepsilon}
\newcommand{\EE}{\mathscr{E}}

\newcommand{\cutoff}{\varphi_{\varepsilon,x_0}}
\newcommand{\M}{\mathsf{M}}
\newcommand{\inner}[2]{\left\langle#1,#2\right\rangle}

\def\Nx{\partial_x}

\def\({\left(}
\def\){\right)}
\def\eb{\varepsilon}

\title[Non-concentration phenomenon]{Non-concentration phenomenon for one dimensional reaction-diffusion systems with mass dissipation}

\author[J. Yang]{Juan Yang}
\address{Juan Yang \hfill\break
	School of Mathematics and Statistics, Lanzhou University Lanzhou, 730000, PR China \; \& \;
	Institute of Mathematics and Scientific Computing, University of Graz,
	Heinrichstrasse 36, 8010 Graz, Austria}
\email{jyang20@lzu.edu.cn}

\author[A. Kostianko]{Anna Kostianko}
\address{Anna Kostianko \hfill\break
	School of Mathematics and Statistics, Lanzhou University Lanzhou, 730000, PR China \; \& \;{Imperial College, London SW7 2AZ, United Kingdom} \; \& \;
	{National Research University Higher School of Economics, Russian Federation}}
\email{a.kostianko@imperial.ac.uk}	

\author[C. Sun]{Chunyou Sun}
\address{Chunyou Sun \hfill\break
	School of Mathematics and Statistics, Lanzhou University Lanzhou, 730000, PR China}
\email{sunchy@lzu.edu.cn}

\author[B.Q. Tang]{Bao Quoc Tang}
\address{Bao Quoc Tang \hfill\break
	Institute of Mathematics and Scientific Computing, University of Graz,
	Heinrichstrasse 36, 8010 Graz, Austria}
\email{quoc.tang@uni-graz.at, baotangquoc@gmail.com}

\author[S. Zelik]{Sergey Zelik}
\address{Sergey Zelik\hfill\break
	University of Surrey, Department of Mathematics, Guildford, GU2
	7XH, United Kingdom \; \&\;
	Keldysh Institute of Applied Mathematics, Moscow, Russia \; \& \;
	{National Research University Higher School of Economics, Russian Federation}} 
\email{s.zelik@surrey.ac.uk}

\begin{document}

\begin{abstract}
%	It is well known, for reaction-diffusion systems with mass dissipation, that the $L^1$-norm is bounded uniformly in time. This is 
	Reaction-diffusion systems with mass dissipation are known to possess blow-up solutions in high dimensions when the nonlinearities have super quadratic growth rates. In dimension one, it has been shown recently that one can have global existence of bounded solutions if nonlinearities are at most cubic. For the cubic intermediate sum condition, i.e. nonlinearities might have arbitrarily high growth rates, an additional entropy inequality had to be imposed. In this article, we remove this extra entropy assumption completely and obtain global boundedness for reaction-diffusion systems with cubic intermediate sum condition. The novel idea is to show a non-concentration phenomenon for mass dissipating systems, that is the mass dissipation implies a dissipation in a Morrey space $\mathsf{M}^{1,\delta}(\Omega)$ for some $\delta>0$. As far as we are concerned, it is the first time such a bound is derived for mass dissipating reaction-diffusion systems. The results are then applied to obtain global existence and boundedness of solutions to an oscillatory Belousov-Zhabotinsky system, which satisfies cubic intermediate sum condition but does not fulfill the entropy assumption. Extensions include global existence mass controlled systems with slightly-super cubic intermediate sum condition.
%	This bound allows us to combine a modified Gagliardo-Nirenberg inequality involving Morrey spaces and an $L^p$-energy method to show that the solution is bounded in sup norm. The results are also extended to mass controlled systems.
\end{abstract}

\maketitle
\noindent{\small{\textbf{Classification AMS 2010:} 35A01, 35A09, 35K57, 35Q92.}}


%\tableofcontents

\section{Introduction and main results}
\subsection{Problem setting and Main result}
Let $\Omega=(0,L)$ for $L>0$. We study in this paper the global existence of bounded solutions to the following reaction-diffusion system for vector of concentrations $u=(u_1,\cdots,u_m),m\geq 1$,
\begin{equation}\label{eq1.1}
  \begin{cases}
    \partial_{t}u_i-d_i\partial_{xx}u_i=f_i(x,t,u), &x\in \Omega,~ t>0,\\
    \partial_xu_i(0,t)=\partial_xu_i(L,t)=0, &t>0,\\
    u_i(x,0)=u_{i,0}(x), &x\in\Omega,
  \end{cases}
\end{equation}
where $d_i>0$ are diffusion coefficients, the initial data are bounded and non-negative, i.e. $u_{i,0}\in L^{\infty}_+(\Omega), \forall i=1,\ldots, m$. The nonlinearities satisfy the following assumptions:

 %   (Bounded, Nonnegative Initial Data).
\begin{enumerate}[label=(A\theenumi),ref=A\theenumi]
	\item\label{A1} (local Lipschitz and quasi-positivity) for all $i=1,\cdots, m$, $f_i:\Omega\times\mathbb{R}_+\times\mathbb{R}_+^m\rightarrow \mathbb{R}$ is locally Lipschitz continuous in the third component and uniformly in the first two components. Moreover, they are quasi-positive,  that is for any $i\in {1,\cdots ,m}$  and any $(x,t)\in\Omega\times\mathbb{R}_+$, it holds
$$f_i(x,t,u)\geq 0\quad \text{provided} \quad u\in \mathbb{R}^m_+ \quad \text{and} \quad u_i=0;$$
	\item\label{A2} (mass dissipation) it holds
	\begin{equation*}
	\sum^m_{i=1}f_i(x,t,u)\leq 0,~~\forall u\in \R_+^m, \; \forall (x,t)\in \Omega\times \R_+;
	\end{equation*}

    \item\label{A3} ($r$-order intermediate sum condition) there exists a lower triangular matrix $A = (a_{ij}) \in \R^{(m-1)\times (m-1)}$ with non-negative elements and positive diagonal elements such that for any $i=1,\ldots, m-1$,
	\begin{equation*}
		\sum_{j=1}^{i}a_{ij}f_j(x,t,u) \leq C\bra{1+\sum_{k=1}^mu_k}^r \quad \forall \;u\in \R_+^m, \; \forall (x,t)\in\Omega\times \R_+,
	\end{equation*}
	for some $r\in [1,\infty)$ where $C>0$ is a fixed constant (when $r=1,2$, or $3$, we refer to this assumption as {\it linear}, {\it quadratic}, or {\it cubic} intermediate sum condition, respectively);
	\item\label{A4} (bound from above by polynomials) there are $\ell>0$ and $C>0$ such that 
	\begin{equation*}
		f_i(t,x,u) \le C\left(1 + \sum_{j=1}^m u_j^\ell \right), \quad \forall i=1,\ldots, m, \; \forall u\in \R_+^m, \; \forall (x,t)\in\Omega\times\R_+.
	\end{equation*}
	It is remarked that \textit{we impose no restriction on the order $\ell$ of the polynomials}.
\end{enumerate}

\medskip
The main result of this paper is the following theorem.
\begin{theorem}[Mass dissipation and cubic intermediate sum]\label{thm:main}
	Assume \eqref{A1}--\eqref{A4}, and $r$ in \eqref{A3} fulfills
	\begin{equation*}
		r \le 3.
	\end{equation*}
	Then for any non-negative bounded initial data $u_0 \in L^\infty_+(\Omega)^m$ there is a unique global classical solution to \eqref{eq1.1} which is bounded uniformly in time, i.e.
	\begin{equation*}
		\sup_{t>0}\sup_{i=1,\ldots, m}\|u_i(t)\|_{L^{\infty}(\Omega)} < +\infty.
	\end{equation*}
\end{theorem}

\begin{remark}\label{rem1}\hfill
    \begin{itemize}
        \item The mass control assumption \eqref{A2} can be straightforwardly replaced by
        \begin{equation*}
            \sum_{i=1}^m \alpha_i f_i(x,t,u) \leq 0, \quad \forall u\in \R_+^m,\; \forall (x,t)\in\Omega\times \R_+,
        \end{equation*}
        for some positive constants $\alpha_i$, $i=1,\ldots, m$.
%        \item The last inequality in \eqref{A3} is in fact redundant since it can be implied from \eqref{A2}.
        \item For simplicity, the initial data in Theorem \ref{thm:main} are assumed to be bounded. However, we believe that $u_{0}\in L^{3}_+(\Omega)^m$ are enough to conclude the Theorem. It remains, however, an interesting open problem if Theorem \ref{thm:main} still holds for only integrable initial data, i.e. $u_0\in L^1_+(\Omega)^m$.
%        \item The (small) constant $\xi$ depends on the distance between the biggest and smallest diffusion coefficients and, as this distance grows to infinity, $\xi$ should decrease to zero. 
    \end{itemize}
\end{remark}
 Extensions of Theorem \ref{thm:main} include the case of mass control and slightly super-cubic intermediate sum condition, which will be stated in the following.
 
\begin{theorem}[Mass control and cubic intermediate sum]\label{thm:mass_control}
	Assume \eqref{A1}, \eqref{A3} with $r=3$, \eqref{A4}, and 
	\begin{equation}\label{mass_control}
		\sum_{i=1}^{m}\alpha_i f_i(x,t,u) \le k_0 + k_1\sum_{i=1}^mu_i, \quad \forall (x,t)\in\Omega \times \R_+,
	\end{equation}
	for some $\alpha_i > 0, i=1,\ldots, m$, $k_0\ge 0$ and $k_1\in \R$. Then for any non-negative bounded initial data $u_0 \in L^\infty_+(\Omega)^m$ there is a unique global classical solution to \eqref{eq1.1}. Moreover, if $k_0 = k_1 = 0$ or $k_1<0$, then the solution is bounded uniformly in time, i.e.
		\begin{equation*}
			\sup_{t>0}\sup_{i=1,\ldots, m}\|u_i(t)\|_{L^{\infty}(\Omega)} < +\infty.
		\end{equation*}
\end{theorem}

\begin{theorem}[Mass control and super-cubic intermediate sum]\label{thm:super_cubic}
	Assume \eqref{A1}, \eqref{mass_control}, and \eqref{A4}. Then there exists a constant $\xi > 0$ depending on diffusion coefficients $d_i$, and $L$, such that if \eqref{A3} holds with 
	\begin{equation*}
		r \le 3 + \xi
	\end{equation*}
	then for any non-negative bounded initial data $u_0 \in L^\infty_+(\Omega)^m$ there is a unique global classical solution to \eqref{eq1.1}. Moreover, if $k_0 = k_1 = 0$ or $k_1<0$, then the solution is bounded uniformly in time, i.e.
			\begin{equation*}
				\sup_{t>0}\sup_{i=1,\ldots, m}\|u_i(t)\|_{L^{\infty}(\Omega)} < +\infty.
			\end{equation*}
\end{theorem}
\begin{remark}
	Obviously, proving Theorem \ref{thm:super_cubic} implies immediately Theorems \ref{thm:main} and \ref{thm:mass_control}. However, we choose to present in detail the proof of Theorem \ref{thm:main} as it shows the main novelty of the new bound in Morrey space. Proofs of Theorems \ref{thm:mass_control} and \ref{thm:super_cubic} follow with some technical but slight changes, which will be shown in Section \ref{sec:extensions}.
\end{remark}

\subsection{State of the art}
The global well-posedness of reaction-diffusion systems of \eqref{eq1.1} under assumptions \eqref{A1} and \eqref{A2} have attracted a lot of attention and recently witnessed considerable progress. On the one hand, conditions \eqref{A1} and \eqref{A2} appear naturally in a great number of models arising from physics, chemistry or biology.
On the other hand, showing global well-posedness under these two assumptions turns out to be challenging since it was shown in \cite{pierre2000blowup} that \eqref{A1} and \eqref{A2} alone are not enough to prevent blow-up in finite time.  To put the present paper into context, let us review related results in the literature. The global bounded solutions to \eqref{eq1.1} with \eqref{A1}, \eqref{A2} and linear intermeidate sum has been investigated already at the end of the eighties  using duality methods , see e.g. \cite{morgan1989global,morgan1990boundedness}. Because of the blow-up example in \cite{pierre2000blowup}, additional assumptions must be imposed. One typical condition is that the nonlinearities are bounded by polynomials, i.e. there exists $p\geq 1$ such that
    \begin{equation}\label{polynomial}
        |f_i(u)| \lesssim 1 + |u|^p, \quad \forall u\in \R_+^m, \; \forall i = 1,\ldots, m.
    \end{equation}
 By assuming the mass conservation condition
    \begin{equation}\label{mass-conservation}
        \sum_{i=1}^mf_i(u) = 0,
    \end{equation}
    (which is \eqref{A2} when the equality is fulfilled) and an
    entropy inequality
    \begin{equation}\label{entropy-inequality}
        \sum_{i=1}^mf_i(u)\log u_i \leq 0,
    \end{equation}
    it was shown in \cite{caputo2009global} by the De Giorgi method that \eqref{eq1.1} is globally well-posed for sub-quadratic nonlinearities, i.e. $p<2$ in \eqref{polynomial}. Still by De Giorgi method, \cite{goudon2010regularity} showed global existence under \eqref{mass-conservation} and \eqref{entropy-inequality} and either $n=1, p = 3$ or $n=2, p = 2$. This was later simplified in \cite{tang2018global} by using a modified Galiardo-Nirenberg inequality.
 Note that in all \cite{caputo2009global,goudon2010regularity,tang2018global}, the entropy inequality plays an important role in the proofs.  By proving an improved duality method,  \cite{canizo2014improved} showed global solutions to \eqref{eq1.1} requiring only \eqref{A1}, \eqref{A2} and quadratic growths in two dimensions, i.e. $n=2, p=2$. Moreover, general $p>2$ can be also treated providing the diffusion coefficients are quasi-uniform.
    Systems with quadratic growth in two dimensions have also been investigated in \cite{pierre2017dissipative}.
The case of quadratic intermediate sum condition in two dimensions has been studied recently in \cite{morgan2020boundedness}. It is remarked that requiring \eqref{A3} with $r=2$ is significantly weaker than requiring \eqref{polynomial} with $p=2$ since the former imposes only quadratic growth of {\it one} nonlinearity, and the rest satisfies a ``good'' cancellation rule, when the latter asks for quadratic bounds of {\it all} nonlinearities.  The case of quadratic growth in higher dimensions, i.e. $p=2$ and $n\geq 3$ has been resolved recently. More precisely, \cite{caputo2019solutions} assumed \eqref{mass-conservation}, \eqref{entropy-inequality}, \eqref{polynomial} with $p=2$ and treated the case of whole space $\R^n$. The work \cite{souplet2018global} still assumed the entropy inequality \eqref{entropy-inequality}, and weaken the mass conservation \eqref{mass-conservation} to mass dissipation \eqref{A2} treated both cases of whole space or bounded domains. The entropy dissipation assumption \eqref{entropy-inequality} has been removed completely in \cite{fellner2020global}, i.e. only \eqref{A1}, \eqref{A2}, and \eqref{polynomial} with $p=2$ are assumed. Moreover, solutions are shown to be bounded uniformly in time, see \cite{fellner2019uniform}.  It is worthwhile to mention that results in these works (partly) utilised ideas from an almost unnoticed paper \cite{kanel1990solvability} which treats \eqref{eq1.1} in the case of whole space $\R^n$ and mass conservation \eqref{mass-conservation}. The quadratic growth turns out to be optimal in higher dimension, see \cite{pierre2022examples}. Concerning higher growth rates of nonlinearities in low dimensions, in the recent paper \cite{sun2021regularity}, global existence was shown under \eqref{A1}, \eqref{A2} and cubic growth in one dimension $n=1, p=3$. If \eqref{A2} is replaced by the entropy inequality \eqref{entropy-inequality}, then cubic growth can be replaced by cubic intermediate sum, i.e. $r=3$ in \eqref{A3}. It is also remarked that the entropy inequality plays an important role when treating the cubic intermediate sum condition. There is also a rich literature dealing with weak solutions \cite{pierre2003weak,desvillettes2007global} or renormalized solutions \cite{fischer2015global,lankeit2021global}. We refer the interested reader to the extensive review paper \cite{pierre2010global} for more details.

\medskip
Having reviewed the existing works in the literature, our main result in Theorem \ref{thm:main} is the first result showing global existence and boundedness of solutions to \eqref{eq1.1} with the cubic intermediate sum condition and assuming only \eqref{A1} and \eqref{A2}. In comparison to the results in \cite{sun2021regularity}, we remove completely the entropy dissipation condition \eqref{entropy-inequality}. This seemingly marginal condition plays a crucial role in the analysis of \cite{sun2021regularity} (as it was also the case for previous works \cite{caputo2009global,goudon2010regularity,tang2018global,caputo2019solutions,souplet2018global}). Therefore, proving Theorem \ref{thm:main} requires new ideas, which we will detail in the next subsection.

\subsection{Key ideas}\label{subsec:keys_ideas}

%	{\LARGE{\textcolor{red}{\textbf{We first show that we can replace the mass dissipation by the mass conservation condition.}}}}
	
An immediate consequence of assumptions \eqref{A1} and \eqref{A2} is that the total mass is bounded for all time, i.e.
\begin{equation*}
    \|u_i\|_{L^{\infty}(0,T;L^1(\Omega))} \leq C(\text{initial data}), \quad  \forall T\in (0,T_{\max}), \forall i=1,\ldots, m,
\end{equation*}
where $(0,T_{\max})$ is the maximal existence interval. Utilising this, it was shown in \cite{morgan1989global,fitzgibbon2021reaction} that \eqref{eq1.1} is globally well-posed under \eqref{A3} of order
\begin{equation}\label{critical}
    r < 1 + \frac 2n
\end{equation}
where $n$ is the spatial dimension. This, interestingly, agrees with the Fujita exponent, see \cite{fujita1966blowing}. The results in \cite{morgan1989global,fitzgibbon2021reaction} therefore just fall short for cubic intermediate sum condition in one dimension. By using an improved duality method in \cite{canizo2014improved} one can also obtain a space-time estimate
\begin{equation*}
    \|u_i\|_{L^{2+\eps}(0,T;L^{2+\eps}(\Omega))} \leq C(T, \text{ initial data}), \quad \forall T\in (0,T_{\max}), \; \forall i=1,\ldots, m,
\end{equation*}
for some $\eps>0$. These estimates allow us to deal with \eqref{A3} of order
\begin{equation*}
    r \leq 1 + \frac{4}{n+2}
\end{equation*}
which is only better than \eqref{critical} when $n\geq 3$, see \cite{morgan2020boundedness}. 

\medskip
{\it Our first key idea} in proving Theorem \ref{thm:main} is to prove {\it a new uniform-in-time estimate} of the form
\begin{equation}\label{key_estimate}
    \|u_i\|_{L^{\infty}(0,T;\M^{1,\delta}(\Omega))} \leq C(\text{initial data}),
\end{equation}
$\forall T\in (0,T_{\max}), \; \forall i=1,\ldots, m$, for some $\delta>0$, where $\M^{1,\delta}(\Omega)$ denotes the Morrey space, see \eqref{def_Morrey}. The estimate \eqref{key_estimate} has sometimes been referred to as the {\it non-concentation phenomenon} which was studied extensively in the context of dispersive equations (see e.g. \cite{bourgain1999global,grillakis2000nonlinear,tao2006nonlinear}). Up to our knowledge, the present paper is the first work proving the estimate \eqref{key_estimate} for mass controlled systems \eqref{eq1.1}. To prove \eqref{key_estimate}, we exploit the recent result shown in \cite{sun2021regularity} that the auxiliary function $\Nx\bra{\int_0^t \sumi d_iu_i(x,s)ds}$ is H\"older continuous. Because of this, Theorem \ref{thm:main}, or more precisely the bound \eqref{key_estimate}, is restricted to the case of one dimension. However, we expect that such a bound is true in any dimension, which is left for future research.

\medskip
Having obtained \eqref{key_estimate},  {\it our second main idea} is to utilise a new version of modified Gagliardo-Nirenberg inequality involving Morrey spaces of the form: for any $\eps>0$, there exists $C(\eps)>0$ such that
\begin{equation}\label{key_estimate_2}
    \|f\|_{L^4(\Omega)}^4 \leq \eps\|f\|_{\M^{1,\delta}(\Omega)}^2\|\Nx f\|_{L^2(\Omega)}^2 + C(\eps)\|f\|_{L^1(\Omega)}^4.
\end{equation}
Note that this inequality is an alternative to the modified Gagliardo-Nirenberg inequality used in \cite{tang2018global,sun2021regularity} where the Morrey space was replaced by $\|u\log u\|_{L^1(\Omega)}$. The interpolation inequality \eqref{key_estimate_2} then allows us to apply the $L^p$-energy estimate developed in \cite{morgan2021global,fitzgibbon2021reaction} to show that
\begin{equation*}
    \sup_{i=1,\ldots, m}\|u_i\|_{L^{\infty}(0,T;L^2(\Omega))} \leq C(\text{initial data}), \quad \forall T\in (0,T_{\max}).
\end{equation*}
This is enough to apply the standard bootstrap procedure and obtain the global existence and boundedness of a unique classical solution to \eqref{eq1.1} in one dimension.

\newcommand{\Sone}{\text{BrO}_3^{-}}
\newcommand{\Stwo}{\text{Br}^{-}}
\newcommand{\Sthree}{\text{H}^{+}}
\newcommand{\Sfour}{\text{HBrO}_2}
\newcommand{\Sfive}{\text{HOBr}}
\newcommand{\Ssix}{\text{BrO}_2'}
\newcommand{\Sseven}{\text{H}_2\text{O}}
\newcommand{\Seight}{\text{Me}^{+}}
\newcommand{\Snine}{\text{Me}^{++}}
\newcommand{\Sten}{\text{Br}_2}
\newcommand{\Seleven}{\text{RH}}
\newcommand{\Stwelve}{\text{RBr}}
\newcommand{\Sthirteen}{\text{R}'}
\newcommand{\Sfourteen}{\text{ROH}}
\newcommand{\Sfifteen}{\text{Br}'}
\subsection{Application}\label{subsec:applications}
	Let us demonstrate an application of Theorem \ref{thm:main} to a realistic system to which all of existing works in the literature cited above, up to our knowledge, are not applicable. Consider the following set of $15$ reactions, which demonstrate the Belousov-Zhabotinsky reaction system (see \cite{fogler5th} and the associated Web modules)
	\begin{center}
	\begin{tabular}{p{8cm} l}
		1. $\quad \Sone + \Stwo + 2\Sthree \xrightarrow{k_1} \Sfour + \Sfive$ & \;\,9. $\quad \Sfive + \Stwo + \Sthree \xrightarrow{k_9} \Sten + \Sseven$ \\
		2. $\quad \Sfour + \Sfive \xrightarrow{k_2} \Sone + \Stwo + 2\Sthree$ &10. $\quad \Sten + \Sseven \xrightarrow{k_{10}} \Sfive + \Stwo + \Sthree$\\
		3. $\quad \Sfour + \Stwo + \Sthree \xrightarrow{k_3} 2\Sfive$ & 		11. $\quad \Seleven + \Sten \xrightarrow{k_{11}} \Stwelve + \Stwo + \Sthree$\\
		4. $\quad \Sone + \Sfour + \Sthree \xrightarrow{k_4} 2\Ssix + \Sseven$ & 		12. $\quad \Sfive + \Sthirteen \xrightarrow{k_{12}} \Sfourteen + \Sfifteen$\\
		5. $\quad 2\Ssix + \Sseven \xrightarrow{k_5} \Sone + \Sfour + \Sthree$ & 		13. $\quad \Seleven + \Sfifteen \xrightarrow{k_{13}} \Stwo + \Sthree + \Sthirteen$\\
		6. $\quad \Ssix + \Seight + \Sthree \xrightarrow{k_6} \Sfour + \Snine$ & 		14. $\quad \Seleven + \Snine \xrightarrow{k_{14}} \Seight + \Sthree + \Sthirteen$\\
		7. $\quad \Sfour + \Snine \xrightarrow{k_7} \Ssix + \Seight + \Sthree$ & 		15. $\quad 2\Sthirteen + \Sseven \xrightarrow{k_{15}} \Seleven + \Sfourteen$\\
		8. $\quad 2\Sfour \xrightarrow{k_8} \Sone + \Sfive + \Sthree$ & \
	\end{tabular}
	\end{center}
	where $\Seleven$ is malonic acid, $\Ssix, \Sthirteen, \Sfifteen$ are free radical intermediates, and $k_{i}, i=1,\ldots, 15$ are positive reaction rate constants. To write down the corresponding reaction-diffusion system, we denote by $u_i(x,t)$, $i=1,\ldots, 15$, the molar concentrations of $\Sone$, $\Stwo$, $\Sthree$, $\Sfour$, $\Sfive$, $\Ssix$, $\Sseven$, $\Seight$, $\Snine$, $\Sten$, $\Seleven$, $\Stwelve$, $\Sthirteen$, $\Sfourteen$, and $\Sfifteen$, respectively.
	% It is noted that the equations $u_{12}$ and $u_{14}$ are decoupled, since $\Stwelve$ and $\Sfourteen$ are products but not reactants.
    Concerning the global existence of solutions, we can assume that $k_i = 1$, $i=1,\ldots, 15$, without affecting the following analysis.  By applying the law of mass action for the reactions and Fickian's law for the diffusion, we obtain the following system of 15 equations 
    \begin{equation}\label{BZ}
    	\begin{cases}
    		\partial_t u_i - d_i\Delta u_i = f_i(u), &x\in \Omega, \; i\in \{1,\ldots, 15\},\\
    		\partial_xu_i(0,t) = \partial_xu_i(L,t) = 0, &i\in\{1,\ldots, 15\},\\
    		u_{i}(x,0) = u_{i,0}(x), &x\in\Omega,\; i\in\{1,\ldots, 15\},
    	\end{cases}
   	\end{equation}
   	where the nonlinearities are given by
	\begin{equation}\label{BZ_sys}
%		\left\{
		\begin{aligned}
			&f_1(u) = -(u_1u_2u_3^2 - u_4u_5) - (u_1u_3u_4 - u_6^2u_7) + u_4^2\\
			& f_2(u) = -(u_1u_2u_3^2 - u_4u_5) - u_2u_3u_4 - (u_2u_3u_5 - u_7u_{10}) + u_{10}u_{11} + u_{11}u_{15}\\
			& f_3(u) = -2(u_1u_2u_3^2 - u_4u_5) - u_2u_3u_4 - (u_1u_3u_4 - u_6^2u_7)\\
			&\qquad\qquad\;\;  - (u_3u_6u_8 - u_4u_9)  - (u_2u_3u_5 - u_7u_{10}) + u_4^2  + u_{10}u_{11} + u_{11}u_{15} + u_9u_{11}\\
			& f_4(u) = u_1u_2u_3^2 - u_4u_5 - u_2u_3u_4 - (u_1u_3u_4 - u_6^2u_7)  - (u_4u_9 - u_3u_6u_8) - 2u_4^2\\
			& f_5(u) = u_1u_2u_3^2 - u_4u_5 + 2u_2u_3u_4 + u_4^2   - (u_2u_3u_5 - u_7u_{10}) - u_5u_{13}\\
			& f_6(u) = 2(u_1u_3u_4 - u_6^2u_7) - (u_3u_6u_8 - u_4u_9)\\
			& f_7(u) = u_1u_3u_4 - u_6^2u_7 + (u_2u_3u_5 - u_7u_{10}) -u_7u_{13}^2\\
			& f_8(u) = -(u_3u_6u_8 -u_4u_9) +u_9u_{11}\\
			& f_9(u) = u_3u_6u_8 - u_4u_9 - u_{9}u_{11}\\
			& f_{10}(u) = u_2u_3u_5 - u_7u_{10} - u_{10}u_{11}\\
			& f_{11}(u) = -u_{10}u_{11} - u_{11}u_{15} - u_9u_{11} + u_7u_{13}^2\\
			& f_{12}(u) = u_{10}u_{11}\\
			& f_{13}(u) = -u_5u_{13} + u_{11}u_{15} + u_{9}u_{11} - 2u_7u_{13}^2\\
			& f_{14}(u) = u_7u_{13}^2 + u_{13}u_{15}\\
			& f_{15}(u) = u_5u_{13} - u_{11}u_{15}.
		\end{aligned}
%		\right.
	\end{equation}
	
	\medskip
	The assumption \eqref{A1} is obviously fulfilled as the nonlinearities are derived from the law of mass action. It can be directly checked that
	\begin{align*} 
		128f_1(u) + 80f_2(u) + f_3(u) + 113f_4(u) + 97f_5(u) + 112f_6(u) + 18f_7(u) + 104f_8(u) + 104f_9(u)\\
		 + 160f_{10}(u) + 104f_{11}(u) + 183f_{12}(u) + 103f_{13}(u) + 120f_{14}(u) + 80f_{15}(u) = 0
	\end{align*}
	where the coefficients are computed from the molar masses of chemical substances in the system. This means that a weighted version of the assumption \eqref{A2} is satisfied (see also Remark \ref{rem1}). To check \eqref{A3}, we notice that all nonlinearites have third or lower order of polynomials, except for $f_2(u)$, $f_3(u)$, $f_4(u)$ and $f_5(u)$. Among these, the fourth order terms in $f_2(u)$ and $f_3(u)$ have negative sign, and the positive fourth order terms  of $f_4(u)$ and $f_5(u)$ are cancelled by e.g. adding $f_3(u)$ and $f_4(u)$ or $f_5(u)$. In short, for all $u\in \mathbb{R}^{15}$,
	\begin{gather*} 
		f_i(u)\lesssim 1 + |u|^3, \quad \forall i=\{1,\ldots, 15\}\backslash\{4,5\},\\
		f_{3}(u) + f_4(u) \lesssim 1 + |u|^3, \qquad f_{3}(u) + f_5(u) \lesssim 1 + |u|^3.
	\end{gather*}
	Therefore, by taking the matrix $A = (a_{i,j})\in \mathbb{R}^{15\times 15}$ with $a_{3,4} = a_{3,5} = 1$ and $a_{i,j} = \delta_{ij}$, the Kronecker Delta, otherwise, we see that the assumption \eqref{A3} is satisfied with $r=3$, i.e. the cubic intermediate sum condition. Moreover, due to the third order reversible reaction (the fourth and fifth ones in the above fifteen reactions), one cannot expect a cubic intermediate sum condition of any lower order.
	 It is important to note that due to the oscillatory behaviour of Belousov-Zhabontinsky reactions \cite{fogler5th}, the entropy inequality \eqref{entropy-inequality} \textit{does not hold}, which makes the result in \cite{sun2021regularity} not applicable. Consequently, our Theorem \ref{thm:main} seems to be the only available result to date which shows global existence of solutions to \eqref{BZ}--\eqref{BZ_sys}.
	\begin{theorem}
		Let $\Omega = (0,L), L>0$. Then for any bounded, non-negative initial data $u_{i,0}\in L^{\infty}(\Omega), i=1,\ldots, 15$, there exists a unique global classical solution to \eqref{BZ}--\eqref{BZ_sys} which is bounded uniformly in time.
	\end{theorem}

\medskip
{\bf Organisation of the paper:} In the next section, we start with some preliminary results which will be needed for the sequel analysis. An easy but important observation is that the mass dissipation condition in \eqref{A2} can be replaced by a mass conservation assumption (see Proposition \ref{imply}). Section \ref{sec:Morrey} is devoted to proving the key estimate \eqref{key_estimate} and the interpolation inequality \eqref{key_estimate_2}. The proof of Theorem \ref{thm:main} (or equivalently, Theorem \ref{thm:mass_conservation}) is presented in Section \ref{sec:second_thm}. The last section is devoted to the proofs of Theorem \ref{thm:mass_control} and \ref{thm:super_cubic}.

\section{Preliminaries}
\subsection{Mass dissipation and Mass conservation}
We first define what we mean a classical solution of \eqref{eq1.1}.
\begin{definition}[Classical  solutions]
	Let $0<T\leq\infty$. A  classical solutions to \eqref{eq1.1} on $(0,T)$ is a value of concentrations $u=(u_1,\cdots,u_m),m\geq1$, satisfying for all $i=1,\cdots,m$, $u_i\in C((0,T);L^p(\Omega))\cap L^\infty((0,T)\times\Omega)\cap C^{1,2}((\tau,T)\times \overline{\Omega})$ for all $p>1$ and all $0<\tau<T$, and u solves \eqref{eq1.1}   a.e. in $\Omega\times(0,T)$.
\end{definition}

The following local existence of \eqref{eq1.1} is classical, see e.g. \cite{morgan1989global}.
\begin{theorem}[Local existence]\label{local}
	Assume \eqref{A1}. Then, for any  bounded, nonnegative initial data, \eqref{eq1.1} possesses a local nonnegative  classical  solution on a maximal interval $[0,T_{\max})$. Moreover, if
	\begin{equation}\label{eq-blow-up}
	\begin{split}
	\limsup_{t\nearrow T_{\max}}\|u_i(t)\|_{L^\infty(\Omega)}<\infty \quad \text{for all}~i=1,2,\cdots,m,
	\end{split}
	\end{equation}
	then $T_{\max}=+\infty.$
	\end{theorem}
	\begin{remark}
		It is noted that the local Lipschitz continuity of the nonlinearities is enough to get local existence. Theorem \ref{local} assumes \eqref{A1} to ensure that if the initial data is non-negative, so is the solution. This is the consequence of the quasi-positivity property.
	\end{remark}

	Thanks to the criterion \eqref{eq-blow-up} in Theorem \ref{local}, the global existence of classical  solutions to \eqref{eq1.1} follows if we can show that
	\begin{equation}\label{criterion}
	\sup_{i}\limsup_{t\nearrow T_{\max}}\|u_i(t)\|_{\LO{\infty}} < +\infty.
	\end{equation}
	
	To prove Theorem \ref{thm:main} we first consider a (seemingly) stronger condition than \eqref{A2},
    \begin{enumerate}[label=(A\theenumi'),ref=A\theenumi']
       \setcounter{enumi}{1}
    	\item\label{A2'} assume that
    	\begin{equation*}
    	  \sum^m_{i=1}f_i(x,t,u) = 0,\qquad \forall u\in \R_+^m, \; \forall (x,t)\in \Omega\times \R_+.
    	\end{equation*}
    \end{enumerate}

    Under \eqref{A2'}, we prove
    \begin{theorem}\label{thm:mass_conservation}
    	Assume \eqref{A1}, \eqref{A2'}, \eqref{A3} with $r = 3$, and \eqref{A4}. Then for any non-negative bounded initial data $u_0 \in L^\infty_+(\Omega)^m$ there is a unique global classical solution to \eqref{eq1.1} which is bounded uniformly in time, i.e.
    	\begin{equation*}
    		\sup_{t>0}\sup_{i=1,\ldots, m}\|u_i(t)\|_{L^{\infty}(\Omega)} < +\infty.
    	\end{equation*}
    \end{theorem}
	The proof of Theorem \ref{thm:mass_conservation} is presented in Section \ref{sec:second_thm}. At the first glance, \eqref{A2'} is stronger than \eqref{A2}. However, the following result shows that Theorem \ref{thm:main} can be implied from Theorem \ref{thm:mass_conservation} by adding a suitable additional equation to the system. This idea has been used before in \cite{fellner2020global,sun2021regularity}.
	\begin{proposition}\label{imply}
		Theorem \ref{thm:mass_conservation} implies  Theorem \ref{thm:main}.
	\end{proposition}
	\begin{proof}	
		Consider system \eqref{eq1.1} with assumptions \eqref{A1}--\eqref{A4}, and let $u = (u_1, \ldots, u_m)$ be the local classical solution obtained in Theorem \ref{local}. Let $u_{m+1}$ be the solution to the parabolic equation on $\Omega \times [0,T_{\max})$
		\begin{equation*}
			\begin{cases}
				\partial_t u_{m+1} - \partial_{xx} u_{m+1} = -\sum_{i=1}^{m}f_i(u), &x\in\Omega, t\in (0,T_{\max}),\\
				\pa_{x}u_{m+1}(0,t) = \pa_{x}u_{m+1}(L,t) = 0, &t\in (0,T_{\max}),\\
				u_{m+1}(x,0) = 0, &x\in\Omega.
			\end{cases}
		\end{equation*}
		Now define 
		\begin{equation*}
			y_j = u_j, \; j=1,\ldots, m+1, \quad y = (y_j)_{j=1,\ldots, m+1}, \quad y_{j,0} = u_{i,0},\; i=1,\ldots, m, \; y_{m+1,0} = 0,
		\end{equation*}
		\begin{equation*}
			g_j(y) = f_j(u), \; i=1,\ldots, m, \quad g_{m+1}(y) = -\sum_{j=1}^{m}f_j(u).
		\end{equation*}
		Then $y$ solves the system
		\begin{equation}\label{new_sys}
			\begin{cases}
				\partial_t y_j - d_j\pa_{xx}y_j = g_j(y), &x\in\Omega, t\in (0,T_{\max}),\\
				\pa_{x}y_{j}(0,t) = \pa_{x}y_{j}(L,t) = 0, &t\in (0,T_{\max}),\\
				y_{j}(x,0) = y_{j,0}(x), &x\in\Omega,
			\end{cases}
		\end{equation}
		for all $j=1,\ldots, m+1$, where $d_{m+1} = 1$. It is easy to check that the nonlinearities $g_j$, $j=1,\ldots, m+1$ satisfy the assumptions \eqref{A1}, \eqref{A2'}, \eqref{A3}, and \eqref{A4}. Thus Theorem \ref{thm:mass_conservation} is applicable and we obtain the unique global classical, uniform-in-time bounded solution $y = (y_1, \ldots, y_{m+1})$ to \eqref{new_sys} with 
		\begin{equation*}
			\limsup_{t\to \infty}\|y_j(t)\|_{L^\infty(\Omega)} < +\infty.
		\end{equation*}
		This implies the global existence, uniqueness and uniform-in-time boundedness of solution to \eqref{eq1.1}, which finishes our proof.
	\end{proof}
	
	Because of Proposition \ref{imply}, we only need to prove Theorem \ref{thm:mass_conservation}. Therefore, \textbf{we will assume \eqref{A2'} for the rest of this paper} except for Section \ref{sec:extensions}.
	
	\subsection{Auxiliary lemmas}
	We will need the following lemmas for our sequel analysis. The first one is an $L^2$-duality estimate. It is noted that this result holds true in any dimension.
	\begin{lemma}\cite{pierre2000blowup,fellner2019uniform}\label{lem:duality}
		Assume \eqref{A1}--\eqref{A2}. Let $u$ be the local solution to \eqref{eq1.1}. Then it holds for any $0\le \tau\le T\in (0,T_{\max})$,
		\begin{equation*}
			\|u_i\|_{L^2(\tau,T;L^2(\Omega))} \leq C(T-\tau,\|u_{i,0}\|_{L^2(\Omega)}), \quad \forall i=1,\ldots, m
		\end{equation*}
		where $C(T-\tau)$ depends polynomially on $T-\tau$.
	\end{lemma}
	The second lemma is the H\"older continuity of an auxiliary function.
	\begin{lemma}\cite{sun2021regularity}\label{lem:Holder}
		Assume \eqref{A1}--\eqref{A2'}. The function $Y$ defined as
		\begin{equation*}
			Y(x,t) = \Nx\sbra{\int_\tau^t\sumi d_iu_i(x,s)ds}, \quad (x,t)\in \Omega\times(\tau,T),
		\end{equation*}
		with $0\le \tau < T < T_{\max}$,
		is H\"older continuous, i.e. there is some $\gamma \in (0,1)$ such that
		\begin{equation*}
			|Y(t,x) - Y(t',x')| \leq \mathscr{C}(T-\tau)\bra{|x - x'|^{2\gamma} + |t - t'|^{\gamma}} \quad \forall (t,x), (t',x') \in \Omega\times(0,T),
			\end{equation*}
			where $\mathscr{C}(T-\tau)$ depends on $T-\tau$ through $\|Y\|_{L^{\infty}(\Omega\times(\tau,T))}\leq C(\Omega, T-\tau, d_i)$.
	\end{lemma}
%	\begin{center}
%		{\color{red}\textbf{Probably the same as the previous lemma.}}
%	\end{center}

\section{Bounds in Morrey spaces}\label{sec:Morrey}
\subsection{Uniform-in-time bound in Morrey spaces}
	For $\delta>0$, the \textit{Morrey norm} is defined as
	\begin{equation}\label{def_Morrey}
		\|f\|_{\M^{1,\delta}(\Omega)}:= \sup_{0<\eps<L}\sup_{x_0\in\Omega}\eps^{-\delta}\|f\|_{L^1([x_0-\eps,x_0+\eps]\cap \Omega)}.
	\end{equation}
	The main result of this section is the following
	\begin{proposition}\label{pro:Morrey_bound}
		Assume \eqref{A1}, \eqref{A2'}, and \eqref{A3}. Then there exist $\delta>0$ and $\mathscr{K}>0$ such that
		\begin{equation}\label{Morrey_bound}
			\sup_{i=1,\ldots, m}\|u_i(T)\|_{\M^{1,\delta}(\Omega)} \leq \mathscr{K}(d_i,\|u_{i0}\|_{L^2(\Omega)}),
		\end{equation}
		where $\mathscr{K}$ is independent ot $T$.
	\end{proposition}

	To prove Proposition \ref{pro:Morrey_bound}, will use the following auxiliary functions
			\begin{equation}\label{auxiliary}
				\begin{aligned}
				z(x,t) = \sumi u_i(x,t), \quad w(x,t) = \sumi d_iu_i(x,t)
				\quad \text{ and }\quad  Y(x,t) = \pa_x\sbra{\int_0^t w(x,s)ds}.
			 	\end{aligned}
			\end{equation}
	With these notation, it follows from summing equations in \eqref{eq1.1} and using \eqref{A2'} that
	\begin{equation}\label{aux_eq}
		\pa_t z - \pa_{xx}w = 0.
	\end{equation}
	To deal with Morrey spaces, let $\varphi\in C^{\infty}(\R)$ such that
	$\varphi(x)\equiv 1$ if  $x\in[-1,1]$ and $\varphi(x)\equiv0$ if $x\notin[-2,2]$.
	For $\eps>0$ and $x_0\in\Omega$, we define the rescaled function
	\begin{equation}\label{def-phi}
	\varphi_{\eb,x_0}(x):=\varphi\(\frac{x-x_0}\eb\).
	\end{equation}
	It's clear that
	\begin{equation}\label{def-phi'}
		|\cutoff'(x)| \leq C\eps^{-1} \quad \forall x\in \R,
	\end{equation}
    where the constant $C$ is independent of $\eps$ and $x_0$.

	Proposition \ref{pro:Morrey_bound} will be proved by considering the case when $T$ is small in Lemma \ref{small_time} and when $T$ is large in Lemma \ref{large_time}.	
	\begin{lemma}\label{small_time}
		Let $0<\eps<L$ ($\eps$ is small) and $x_0\in\Omega$. Then for $0\leq T \leq \eps^{\frac{1}{1+2\gamma}}$ we have
		\begin{equation}\label{f1}
			\intO z(x,T)\cutoff(x)dx \leq C\bra{1+\|z(0)\|_{L^2(\Omega)}}\eps^{\frac{\gamma}{1+2\gamma}},
		\end{equation}
		where the constant $C$ is independent of $\eps$, $T$ and $x_0$.
	\end{lemma}
	\begin{proof}
		From \eqref{aux_eq}, direct computations give		
\begin{align*}
			\intO z(x,T)\cutoff dx - \intO z(x,0)\cutoff dx
			&= \intO \cutoff \pa_{xx}\bra{\int_{0}^T w(x,s)ds}dx\\%  + \intO\int_{0}^T\cutoff g(x,t)dxdt\\
			&=-\intO  \pa_x\bra{\int_0^T w(x,s)ds}\pa_x\cutoff dx\\%+ \intO\int_{0}^T\cutoff g(x,t)dxdt\\
			&= -\intO [Y(T,x) - Y(0,x)]\pa_x\cutoff  dx \leq C\|\pa_x\cutoff\|_{L^1(\Omega)}T^{\gamma},
		\end{align*}
		where we apply the H\"older continuity of $Y$ proved in Lemma \ref{lem:Holder} at the last step.
		From the definition of $\cutoff$ we have
		\begin{equation*}
			\|\pa_x \cutoff\|_{L^1(\Omega)} = \eps^{-1}\int_{\Omega}\left|\varphi'\bra{\frac{x-x_0}{\eps}} \right|dx \leq \eps^{-1}\int_{\Omega\cap [x_0-2\eps,x_0+2\eps]}C_\varphi dx \leq 4C_\varphi.
		\end{equation*}
		Thus it follows
		\begin{equation*}
			\intO z(x,T)\cutoff(x)dx \leq \intO z(x,0)\cutoff(x)dx + C(\varphi,L)T^{\gamma}.
		\end{equation*}
		By H\"older's inequality
		\begin{align*}
			\intO z(x,0)\cutoff(x)dx &\leq \|z(0)\|_{L^2(\Omega)}\sbra{\int_{\Omega\cap [x_0-2\eps,x_0+2\eps]}|\cutoff(x)|^2dx}^{1/2}\\
			&\le 2\|z(0)\|_{L^2(\Omega)}\eps^{1/2}.
		\end{align*}
		Therefore, we obtain
		\begin{align*}
			\intO z(x,T)\cutoff(x)dx &\leq 2\|z(0)\|_{L^2(\Omega)}\eps^{1/2} + C(\varphi,L)T^{\gamma}\\
			&\leq \sbra{2\|z(0)\|_{L^2(\Omega)}L^{\frac{1}{2(1+2\gamma)}} + C(\varphi,L)}\eps^{\frac{\gamma}{1+2\gamma}},
		\end{align*}
		which is the desired estimate.
	\end{proof}

	\begin{lemma}\label{large_time}
		Let $0<\eps<L$ and $x_0\in \Omega$. Then for $\eps^{\frac{1}{1+2\gamma}} \le T\leq T_{\max}$ it holds
		\begin{equation*}
			\intO z(x,T)\cutoff(x)dx \leq C(\varphi,\|z\|_{L^2(0,T;L^2(\Omega))})\eps^{\frac{\gamma}{1+2\gamma}},
		\end{equation*}
		where $C$ is independent of $x_0$, $T$ and $\eps$.
	\end{lemma}
	\begin{proof}
		Choose
		\begin{equation}\label{f1_1}
			\tau = T - \eps^{\frac{1}{1+2\gamma}}.
		\end{equation}
		From \eqref{aux_eq} it follows that
		\begin{equation*}
			\pa_t\sbra{(t-\tau)z\cutoff} = z\cutoff + (t-\tau)\cutoff \pa_{xx}w.
		\end{equation*}
		Integration on $\Omega\times (\tau,T)$ gives
		\begin{equation}\label{f2}
		\begin{aligned}
			&(T-\tau)\intO z(x,T)\cutoff(x)dx
			 &= \int_{\tau}^{T}\intO z(x,t)\cutoff(x)dx + \int_{\tau}^T\intO \sbra{(t-\tau)\cutoff(x)\pa_{xx}w}dxdt.
		\end{aligned}
		\end{equation}
		By H\"older's inequality
		\begin{equation}\label{f3}
		\begin{aligned}
			\int_\tau^T\int_{\Omega}z(x,t)\cutoff(x)dxdt
			&\le \int_{\tau}^{T}\int_{\Omega\cap[x_0-2\eps,x_0+2\eps]}z(x,t)dxdt\\
			&\le \sbra{\int_{\tau}^T\intO|z(x,t)|^2dxdt}^{1/2}\sbra{\int_{\tau}^{T}\int_{\Omega\cap[x_0-2\eps,x_0+2\eps]}dxdt}^{1/2}\\
			&\leq \|z\|_{L^2(\tau,T;L^2(\Omega))}\sbra{4\eps|T-\tau|}^{1/2}\\
			&\le C(\eps)\eps^{1/2}|T-\tau|^{1/2} \le C(L)\eps^{1/2}|T-\tau|^{1/2} 
		\end{aligned}
		\end{equation}
		where we apply Lemma \ref{lem:duality} at the last step with the remark that $C(\eps)$ depends polynomially on $\eps$.
		From the auxiliary functions \eqref{auxiliary} we have
		\begin{equation}\label{f3_1}
		\begin{aligned}
			\int_{\tau}^T\intO \sbra{(t-\tau)\cutoff(x)\pa_{xx}w}dxdt
			&= -\int_\tau^T\intO (t-\tau)\pa_x \cutoff \pa_x w(x,t)dxdt\\
			&= -\intO \pa_x \cutoff \bra{\int_{\tau}^T(t-\tau)\pa_t Y(x,t)dt}dx\\
			&= -\intO \pa_x \cutoff\sbra{(T-\tau)Y(x,T)- \int_\tau^T Y(x,t)dt}dx\\
			&= -\intO \pa_x \cutoff\sbra{\int_\tau^T (Y(x,T) - Y(x,t))dt}dx\\
			&\leq \|\pa_x \cutoff\|_{L^1(\Omega)}\int_{\tau}^T\|Y(\cdot,T)-Y(\cdot,t)\|_{L^{\infty}(\Omega)}dt\\
			&\leq C_{\varphi} \int_{\tau}^TC|T - t|^{\gamma}dt \quad (\text{thanks to Lemma \ref{lem:Holder}})\\
			&\leq C(\varphi,\gamma)|T-\tau|^{1+\gamma}.
		\end{aligned}
		\end{equation}
%        Since
%        \begin{equation}\label{f3_2}
%		\begin{aligned}
%			\int^T_{\tau}\int_{\Omega}(t-\tau)g\cutoff dxdt &\leq \|g\|_{\LQ{\infty}}\int^T_{\tau}\int_{\Omega}(t-\tau)\cutoff\, dxdt\\
%            &\leq\|g\|_{\LQ{\infty}}\int^T_{\tau}\int_{\Omega\cap[x_0-2\eps,x_0+2\eps]}(t-\tau)\, dxdt\\
%            &=\frac 12\|g\|_{\LQ{\infty}}(T-\tau)^2\cdot 4\varepsilon\\
%            &=2\varepsilon\|g\|_{\LQ{\infty}}(T-\tau)^2
%		\end{aligned}
%		\end{equation}

		Combining \eqref{f3} \eqref{f2} and  \eqref{f3_1} gives
		\begin{equation*}
        \begin{aligned}
			(T-\tau)\intO z(x,T)\cutoff(x)dx \leq C(L)\sbra{\eps|T-\tau|}^{1/2} + C(\varphi,\gamma)|T-\tau|^{1+\gamma}.
		\end{aligned}
        \end{equation*}
		Due to \eqref{f1_1} we finally obtain
		\begin{equation*}
			\intO z(x,T)\cutoff(x)dx \leq C(\varphi,\gamma,L)\eps^{\frac{\gamma}{1+2\gamma}},
		\end{equation*}
		which finishes the proof of Lemma \ref{large_time}.
	\end{proof}
	
	We are now ready to prove Proposition \ref{pro:Morrey_bound}.
	
	\begin{proof}[Proof of Proposition \ref{pro:Morrey_bound}]
		Define
		\begin{equation}\label{delta}
			\delta = \frac{\gamma}{1+2\gamma},
		\end{equation}
		recalling that $\gamma$ is the H\"older exponent in Lemma \ref{lem:Holder}. 	For any $\eps\in (0,L)$ ($\eps$ is small) and any $x_0\in \Omega$, we have from Lemmas \ref{lem:duality}, \ref{small_time} and \ref{large_time} that
		\begin{equation*}
			\eps^{-\delta}\int_{\Omega\cap [x_0-\eps,x_0+\eps]}z(x,T)dx \leq \eps^{-\delta}\intO z(x,T)\cutoff(x)dx \leq C
		\end{equation*}
		where $C$ is independent of $\eps$, $T$ and $x_0$. By definition of Morrey space we obtain
		\begin{equation*}
			\|z(\cdot,T)\|_{\M^{1,\delta}(\Omega)} \leq C
		\end{equation*}
		for all $T\in (0,T_{\max})$.
	\end{proof}

	\subsection{Interpolation inequalities involving Morrey spaces}
	Another tool to prove our main results is an interpolation inequality with Morrey space. To do that, we need a preparation.
	\begin{lemma}\label{lem-key1}
		Let $\Omega_0 \subset \Omega$ and  $\varphi: \R \to [0,1]$ be a smooth cut-off function such that $\operatorname{supp}\varphi\subset\Omega_0$. Assume also that the following inequality holds:
		\begin{equation}\label{eq-key1-1}
		|\Nx\varphi(x)|\le C_\varphi[\varphi(x)]^{1/3},\ \ x\in\R.
		\end{equation}
		Then
		\begin{equation}\label{eq-key1-2}
		\intO \varphi^2|u|^4dx \leq C\|u\|_{L^1(\Omega_0)}^2\bra{\intO \varphi^2|\Nx u|^2dx+C_\varphi^3\|u\|^2_{L^1(\Omega_0)}},
		\end{equation}
		where the constant $C$ is independent of $\varphi$ and $\Omega_0$.
	\end{lemma}

	\begin{proof}
		We start with the classical interpolation between $L^2$ and $H^1$. Namely, by Newton-Leibnitz,
		$$
		\varphi^{4/3}(x)|u(x)|^2=2\int_{-\infty}^x \varphi^{4/3}(x)u(x)u'(x)\,dx+\frac43\int_{-\infty}^x\varphi^{1/3}(x)\varphi'(x)|u(x)|^2\,dx,
		$$
		and, by the Cauchy-Schwarz inequality together with \eqref{eq-key1-1},
		\begin{equation*}
		\|\varphi^{4/3}u^2\|_{L^\infty(\Omega_0)}\le 2\|\varphi \Nx u\|_{L^2(\Omega_0)}\|\varphi^{1/3}u\|_{L^2(\Omega_0)}+\frac43C_\varphi \|\varphi^{1/3}u\|^2_{L^2(\Omega_0)}.
		\end{equation*}
		
		At the next step, we use the obvious estimate
		\begin{equation*}
		\intO \varphi^{2/3}u^2 dx\le \|u\|_{L^1(\Omega_0)}\|\varphi^{2/3}u\|_{L^\infty(\Omega_0)}
		\end{equation*}
		in order to get
		\begin{equation*}
		\begin{split}
		\|\varphi^{4/3}u^2\|_{L^\infty(\Omega_0)}&\le 2\(\|u\|_{L^1(\Omega_0)}\|\varphi^{4/3}u^2\|_{L^\infty(\Omega_0)}^{1/2}\)^{1/2}\|\varphi\Nx u\|_{L^2(\Omega_0)}\\
		&\quad+\frac43 C_\varphi\|u\|_{L^2(\Omega_0)}\|\varphi^{4/3}u^2\|_{L^\infty(\Omega_0)}^{1/2}.
		\end{split}
		\end{equation*}
		
		This gives the interpolation inequality
		\begin{equation*}
		\|\varphi^{2/3}u\|_{L^\infty(\Omega)}^2\le C\(\|u\|_{L^1(\Omega_0)}^{2/3}\|\varphi\Nx u\|_{L^2(\Omega_0)}^{4/3}+C_\varphi^2\|u\|_{L^1(\Omega_0)}^2\).
		\end{equation*}
		Finally, we have
		\begin{equation*}
		\intO \varphi^2u^4 dx \leq \|u\|_{L^1(\Omega_0)}\|\varphi^{2/3}u\|_{L^\infty(\Omega_0)}^3,
		\end{equation*}
		which gives \eqref{eq-key1-2} and finishes the proof of the Lemma.
	\end{proof}
	The following lemma plays a crucial role in proving our analysis.
	\begin{proposition}[Interpolation inequality with Morrey spaces]\label{pro:interp_Morrey}
		Let $\delta>0$. Then for any $\eps>0$, there exists $C(\eps,\delta)>0$ such that
		\begin{equation*}
			\intO |u|^4dx \leq \eps \|u\|_{\M^{1,\delta}(\Omega)}^2\|\Nx u\|_{L^2(\Omega)}^2 + C(\eps,\delta)\|u\|_{L^1(\Omega)}^4.
		\end{equation*}
	\end{proposition}
	\begin{proof}
		We present the proof for the case $u\in H_0^1(\Omega)$, since for $u\in H^1(\Omega)$ it can be treated by using the extension operator from $H^1(\Omega)$ to $H^1_0(2\Omega)$. There exist $x_1, x_2,\ldots, x_k\in \Omega$ such that \begin{equation*}
			\Omega = \bigcup_{j=1,\ldots, k}\Omega_{j} \quad\text{with}\quad \Omega_j := \Omega\cap (x_j-\eps,x_j+\eps).
		\end{equation*}
		Let $\varphi_j:\Omega\to [0,1]$ such that $\mathrm{supp}(\varphi_j)\subset \Omega_j$ and $			\sum_{j=1}^{k}\varphi_j^2(x) = 1, \forall x\in\Omega$.
		Moreover, we assume that for each $j\in \{1,\ldots, k\}$ there is $C_j = C_j(\eps)>0$ with
		\begin{equation*}
			|\Nx \varphi_j(x)| \leq C_j[\varphi_j(x)]^{1/3} \quad \forall x\in R.
		\end{equation*}
		Now we can apply Lemma \ref{lem-key1} to estimate
		\begin{equation}\label{f4}
		\begin{aligned}
			\intO |u|^4dx &= \sum_{j=1}^{k}\intO \varphi_j^2|u|^4dx\leq C\sum_{j=1}^{k}\sbra{\|u\|_{L^1(\Omega_j)}^2 \intO \varphi_j^2|\Nx u|^2dx + C_j^3\|u\|_{L^1(\Omega_j)}^4}.
		\end{aligned}
		\end{equation}
		On the one hand
		\begin{equation*}
			\|u\|_{L^1(\Omega_j)}^2 = \eps^{2\delta}\sbra{\eps^{-\delta}\int_{\Omega_j} |u|dx}^2\leq \eps^{2\delta}\|u\|_{\M^{1,\delta}(\Omega)}^2,
		\end{equation*}
		which implies
		\begin{align*}
			\sum_{j=1}^k\|u\|_{L^1(\Omega_j)}^2\intO \varphi_j^2 |\Nx u|^2dx &\leq \eps^{2\delta}\|u\|_{\M^{1,\delta}(\Omega)}^2\intO |\Nx u|^2\sum_{j=1}^{k}\varphi_j^2dx \eps^{2\delta}\|u\|_{\M^{1,\delta}(\Omega)}^2\intO |\Nx u|^2dx.
		\end{align*}
		On the other hand
		\begin{equation*}
			\sum_{j=1}^{k}C_j^3\|u\|_{L^1(\Omega_j)}^4 \leq C(\eps,k)\|u\|_{L^1(\Omega)}^4.			
		\end{equation*}
		Therefore, we finally obtain from \eqref{f4} that
		\begin{equation*}
			\intO |u|^4dx \leq C\eps^{2\delta}\|u\|_{\M^{1,\delta}(\Omega)}^2 \|\Nx u\|^2 + C(\eps,k)\|u\|_{L^1(\Omega)}^4,
		\end{equation*}
		which finishes the proof of Proposition \ref{pro:interp_Morrey}.
	\end{proof}

\section{Proof of Theorem \ref{thm:mass_conservation}}\label{sec:second_thm}
Following \cite{sun2021regularity,fitzgibbon2021reaction}, we consider the $L^p$-energy function for $p\in \mathbb N$,
\begin{equation}\label{EE}
	\EE_p[u]:= \sum_{\beta\in \mathbb Z_+^m,|\beta| = p}\begin{pmatrix} p\\\beta\end{pmatrix}\theta^{\beta^2}\intO \prod_{i=1}^mu_i^{\beta_i}dx
\end{equation}
where $\theta \in (0,\infty)^m$ are to-be-chosen coefficients, and we use the notation
\begin{equation*}
	\begin{pmatrix} p\\ \beta \end{pmatrix} = \frac{p!}{\beta_1!\beta_2!\cdots \beta_m!}, \quad \text{ and } \quad \theta^{\beta^2} = \prod_{i=1}^{m}\theta_i^{\beta_i^2}.
\end{equation*}
Thanks to the non-negativity of $u_i$, there exists $\lambda = \lambda(p,\theta)>0$ such that
\begin{equation}\label{norm_equivalent}
	\lambda^{-1}\sumi \|u_i\|_{L^p(\Omega)}^p \leq \EE_p[u] \leq \lambda \sumi \|u_i\|_{L^p(\Omega)}^p.
\end{equation}
The following Lemma was proved in \cite[Lemma 2.5]{fitzgibbon2021reaction}.
\begin{lemma}\label{lem:energy}
	Assume \eqref{A1} and \eqref{A3}. Then for any $2\leq p \in \mathbb{Z}$, there exists $\theta \in (0,\infty)^m$ such that the function $\EE_p[u]$ defined in \eqref{EE} satisfies
	\begin{equation*}
		\frac{d}{dt}\EE_p[u(t)] + \alpha_p\sumi \intO \left|\Nx(u_i(t)^{p/2}) \right|^2dx \leq C\bra{1 + \sumi \intO u_i(t)^{p-1+r}dx}
	\end{equation*}
	with $u$ is the solution to \eqref{eq1.1}, where $\alpha_p, C> 0$.
\end{lemma}

%\begin{lemma}\label{lem:L2}
%	Assume \eqref{A1}, \eqref{A2'} and \eqref{A3} with $r = 3$. Let $u$ be the solution to \eqref{eq1.1}. Then,
%	\begin{equation*}
%		\|u_i\|_{L^p(0,T;L^p(\Omega))} \leq C(T, p, \text{ initial data})
%		%\|u_i\|_{L^{\infty}(0,T;L^2(\Omega)) \cap L^2(0,T;H^1(\Omega))} \leq C(T, \text{ initial data}),
%	\end{equation*}
%	for all $p\geq 1$, for all $ T\in (0,T_{\max})$ and all $i=1,\ldots, m$.
%\end{lemma}
\begin{proof}[\textbf{Proof of Theorem \ref{thm:mass_conservation}}]%: Global existence
	We choose $p=2$ in Lemma \ref{lem:energy}, noting that $r=3$, to get
	\begin{equation}\label{f5}
		\frac{d}{dt}\EE_2[u(t)] + \alpha_2\sumi\|\Nx u_i\|_{L^2(\Omega)}^2 \leq C_0\bra{1 + \sumi \|u_i\|_{L^4(\Omega)}^{4}}
	\end{equation}
	with some $C_0>0$. Thanks to Propositions \ref{pro:Morrey_bound} and \ref{pro:interp_Morrey}, by choosing $\eps>0$ small enough we have
	\begin{equation*}
		\|u_i\|_{L^4(\Omega)}^4 \leq \eps \|u_i\|_{\M^{1,\delta}(\Omega)}^2 \|\Nx u_i\|_{L^2(\Omega)}^2 + C(\eps)\|u_i\|_{L^1(\Omega)}^4 \leq \frac{\alpha_2}{2C_0}\|\Nx u_i\|_{L^2(\Omega)}^2 + C(\eps,\delta).
	\end{equation*}
	Inserting this into \eqref{f5} gives
	\begin{equation*}
		\frac{d}{dt}\EE_2[u(t)] + \frac{\alpha_2}{2}\sum_{i=1}^{m}\|\Nx u_i(t)\|_{L^2(\Omega)}^2 \leq C(\eps,\delta).
	\end{equation*}
	By adding both sides with $\frac{\alpha_2}{2}\sum_{i=1}^m\|u_i\|_{L^2(\Omega)}^2$ and using the interpolation inequalities $\|u_i\|_{L^2(\Omega)}^2 \le \frac{\alpha_2}{2}\|\pa_xu_i\|_{L^2(\Omega)}^2 + C(\alpha_2)\|u_i\|_{L^1(\Omega)}^2$, we obtain
	\begin{equation*}
		\frac{d}{dt}\EE_2[u(t)]+\frac{\alpha_2}{2}\sum_{i=1}^m\|u_i\|_{L^2(\Omega)}^2 \le C.
	\end{equation*}
	Now \eqref{norm_equivalent} and the classical Gronwall lemma give
	\begin{equation*}
		\|u_i\|_{L^{\infty}(0,T;L^2(\Omega))} \leq C(\eps, \delta) \quad \forall i=1,\ldots, m.
	\end{equation*}
	The global existence and uniform-in-time boundedness of a unique classical  solution then follows from \cite[Theorem 1.3]{fitzgibbon2021reaction}.
\end{proof}

\section{Extensions}\label{sec:extensions}
\subsection{Proof of Theorem \ref{thm:mass_control}}\label{sec:first_thm}
\begin{proof}
 We employ the idea from \cite{fellner2020global} with the change of variables
\begin{equation*}
  \begin{split}
    y_i(x,t)=e^{-k_1t}u_i(x,t) \quad \text{or equivalently}  \quad u_i(x,t)=e^{k_1t}y_i(x,t).
  \end{split}
\end{equation*}

Then the system \eqref{eq1.1} gives
\begin{equation*}
  \begin{split}
    \partial_ty_i=d_i \pa_{xx}y_i + e^{-k_1t}(f_i(x,t,u)-k_1e^{k_1t}y_i):= d_i\pa_{xx} y_i+g_i(x,t,y).
  \end{split}
\end{equation*}

It is clear that
\begin{equation*}
  \begin{split}
    \sum^m_{i=1}g_i(x,t,y)=e^{-k_1t}\sum^m_{i=1}(f_i(x,t,u)-k_1u_i)\leq e^{-k_1t}k_0
  \end{split}
\end{equation*}
thanks to \eqref{mass_control}. The global existence of this system now follows from Theorem \ref{thm:main} with slight modifications. Changing back to variables $u_i(x,t) = e^{k_1t}y_i(x,t)$, we obtain the global existence of \eqref{eq1.1}. Note that we do not have yet the uniform-in-time bounds of $y_i$ but only a bound which grows at most polynomial in time (see e.g. \cite{fellner2020global}). Therefore, in the case $k_1<0$, we get the uniform-in-time bounds
\begin{equation*}
	\limsup_{t\to\infty}\|u_i(t)\|_{L^\infty(\Omega)} <+\infty \quad \forall i=1,\ldots, m.
\end{equation*}
The case $k_0 = k_1 =0$ is included in Theorem \ref{thm:main}.
%
%
%Let $y_{m+1}$ be the solution to
%\begin{equation}\label{eq-m+1}
%  \begin{cases}
%  	\pa_t y_{m+1} - \pa_{xx}y_{m+1} = g_{m+1}(x,t,y):= k_0e^{-k_1t} - \sumi g_i(x,t,y) \geq 0,\\
%  	\pa_x y_{m+1}(0,t) = \Nx y_{m+1}(L,t) = 0,\\
%  	y_{m+1}(x,0) = 0.
%  \end{cases}
%\end{equation}
%By writing $\widehat{y} = (y_1,\ldots, y_m, y_{m+1})$ we get the following system
%\begin{equation}\label{eq-y}
%  \begin{cases}
%    \pa_t y_i-d_i\pa_{xx}y_i=g_i(x,t,\widehat{y}), &i=1,\ldots, m+1,\\
%    \pa_xy_i(0,t)=\pa_xy_i(L,t)=0, &i=1,\ldots, m+1,\\
%    y_i(x,0)=u_{i,0}(x), &i=1,2,\cdots,m,\\
%    y_{m+1}(x,0)=0.
%  \end{cases}
%\end{equation}
%It is straightforward to check that \eqref{eq-y} satisfies \eqref{A1}, \eqref{A2'} and \eqref{A3} with $r = 3$. Therefore, Theorem \ref{thm:mass_conservation} ensures that \eqref{eq-y} has a unique global bounded solution. This implies the global existence  of classical solution to \eqref{eq1.1}.
%
%\medskip
%If $k_0 = k_1 = 0$ in \eqref{A2}, then $y_i(x,t) = u_i(x,t)$. Moreover, it holds
%\begin{equation*}
%	\sum_{i=1}^{m+1}g_i(x,t,\widetilde{y}) = 0.
%\end{equation*}
%Thus, Theorem \ref{th_super-cub_equation} gives
%\begin{equation*}
%	\sup_{t\geq 0}\|u_i(t)\|_{L^{\infty}(\Omega)} = \sup_{t\geq 0}\|y_i(t)\|_{L^{\infty}(\Omega)} < +\infty, \forall i=1,\ldots, m,
%\end{equation*}
%which finishes the proof of Theorem \ref{th_super-cubic_mass_control}.
\end{proof}

\subsection{Proof of Theorem \ref{thm:super_cubic}}
%It remains to prove Theorem \ref{thm:super_cubic}.

\begin{lemma}\label{lem-key2}
Let $\varphi: \R \to [0,1]$ be a smooth function such that $\operatorname{supp}\varphi\subset\Omega_0\subset \Omega$ and the following inequality holds
\begin{equation}\label{e1}
|\Nx\varphi(x)|\le C_\varphi[\varphi(x)]^{\frac{1+\delta}{3+\delta}},\quad x\in\R.
\end{equation}
Then
\begin{equation}\label{eq-key2-2}
\inner{\varphi^2}{|u|^{4+\delta}}\le C\|u\|_{L^1(\Omega_0)}^{2-\delta}\| u\|^{2\delta}_{L^2(\Omega_0)}\inner{\varphi^2}{|\Nx u|^2}+CC_\varphi^{3+\delta} \|u\|_{L^1(\Omega_0)}^{4+\delta},
\end{equation}
where the constant $C$ is independent of $\varphi$ and $\Omega_0$. Hereafter, we denote by $\langle\cdot,\cdot\rangle$ the usual inner product of $L^2(\Omega)$.
\end{lemma}
\begin{proof}
	The proof is similar to that Lemma \ref{lem-key1}. It is clear that
	\begin{equation*}
		\inner{\varphi^2}{u^{4+\delta}} \le \|u\|_{L^1(\Omega_0)}\|\varphi^{\frac{2}{3+\delta}}u\|_{L^{\infty}(\Omega_0)}^{3+\delta}.
	\end{equation*}
	Now by the Gagliardo-Nirenberg inequality,
	\begin{equation*}
		\|\varphi^{\frac{2}{3+\delta}}u\|_{L^\infty(\Omega_0)}^2\le C\(\|u\|_{L^1(\Omega_0)}^{\frac{2(1-\delta)}{3+\delta}}\|\varphi\Nx u\|^{\frac{4}{3+\delta}}_{L^2(\Omega_0)}\|u\|^{\frac{4\delta}{3+\delta}}_{L^2(\Omega_0)}+C_\varphi^2\|u\|_{L^1(\Omega_0)}^2\),
	\end{equation*}
	we get \eqref{eq-key2-2} immediately.
\end{proof}
\begin{proof}[Proof of Theorem \ref{thm:super_cubic}]
	Similar to the proof of Theorem \ref{thm:mass_conservation}, it is enough to show that the local solution satisfies $\|u_i\|_{L^{\infty}(0,T;L^2(\Omega))} \le C(T)$ for all $T\in (0,T_{\max})$ with $C(T)$ is continuous in $T$ for $T\in \R_+$. For $\delta$ defined in \eqref{delta}, we choose $\xi > 0$ small enough such that
	\begin{equation}\label{xi}
		\delta > \frac{(3+\xi)\xi}{2-\xi}.
	\end{equation}
	Let $\eps>0$ be chosen later and $x_0 \in \Omega$ arbitrary. Define by $\cutoff: \R \to [0,1]$ the cut-off function with $|\text{supp}(\cutoff)| \le 2\eps$, $\cutoff$ satisfies the assumption \eqref{e1}
	\begin{equation}\label{b0}
		|\cutoff'| \le C\eps^{-1} \quad \text{ and } \quad |\cutoff| \le C\eps^{\frac{1+\delta}{3+\delta}}|\cutoff'|
	\end{equation}
	with $C$ is independent of $\eps$. By chain rule, it follows that
	\begin{equation*}
		\partial_t(u_i\cutoff) - d_i\pa_{xx}(u_i\cutoff) = \cutoff f_i(x,t,u) + d_i\cutoff' \pa_x u_i + d_i\pa_x(\cutoff' u_i).
	\end{equation*}
	Using the $L^p$-energy function in \eqref{EE} with $p=2$, and applying Lemma \eqref{lem:energy} we have
	\begin{equation}\label{b1}
	\begin{aligned}
		\frac{d}{dt}\EE_2[u\cutoff] &+ \alpha_2\sum_{i=1}^m\|\pa_x(u_i\cutoff)\|_{L^2(\Omega)}^2\\
		&\le C\bra{1+\sum_{i=1}^m\inner{\cutoff^2}{|u_i|^{4+\delta}} + \sum_{i=1}^{m}\bra{\inner{u_i\cutoff}{\pa_xu_i \cutoff'} + \inner{|\cutoff'|^2}{u_i^2} }}.
	\end{aligned}
	\end{equation}
	For the right hand side of \eqref{b1}, we apply Lemma \ref{lem-key2} to the first sum, apply Cauchy-Schwarz and \eqref{b0} to the second and the third sums, to get
	\begin{equation}\label{b3}
	\begin{aligned}
		&\frac{d}{dt}\EE_2[u\cutoff]  + \frac{\alpha_2}{2}\sum_{i=1}^{m}\|\pa_x(u_i\cutoff)\|_{L^2(\Omega)}^2\\
		&\le C\eps^{(2-\xi)\delta}\sum_{i=1}^m\|u_i\|_{\M^{1,\delta}(\Omega)}^{2-\xi}\|u_i\cutoff\|_{L^2(\Omega)}^{2\xi}\inner{\cutoff^2}{|\pa_xu_i|^2} + C\eps^{-2}\sum_{i=1}^m\|u_i\|_{L^2(\Omega)}^2 + C\eps^{-(3+\xi)}\sum_{i=1}^m\|u_i\|_{L^1(\Omega)}^{4+\xi}\\
		&\le C\eps^{(2-\xi)\delta}\sum_{i=1}^m\|u_i\cutoff\|_{L^2(\Omega)}^{2\xi}\inner{\cutoff^2}{|\pa_xu_i|^2} + C\eps^{-2}\sum_{i=1}^m\|u_i\|_{L^2(\Omega)}^2 + C\eps^{-(3+\xi)},
	\end{aligned}
	\end{equation}
	where all constants $C$ are independent of $\eps$ and $T$. Choose $\eps_0>0$ such that
	\begin{equation*}
		C\eps^{(2-\xi)\delta}\sum_{i=1}^m\|u_{i,0}\cutoff\|_{L^2(\Omega)}^{2\xi} \le \frac{\alpha_2}{4} \quad \forall \eps\in (0,\eps_0).
	\end{equation*}
	Let $\eps\in (0,\eps_0)$ and define $T_0 = T_0(\eps)$ as 
	\begin{equation*}
		T_0:= \sup\left\{s\in (0,T_{\max}): \; C\eps^{(2-\xi)\delta}\sum_{i=1}^m\|u_{i}(r)\cutoff\|_{L^2(\Omega)}^{2\xi} < \frac{\alpha_2}{2}, \; \forall r\in (0,s) \right\}.
	\end{equation*}
	We show that $T_0 \ge T$ as $\eps$ is small enough. Indeed, by assuming otherwise, we have
	\begin{equation}\label{b2}
		C\eps^{(2-\xi)\delta}\sum_{i=1}^m\|u_{i}(T_0)\cutoff\|_{L^2(\Omega)}^{2\xi} = \frac{\alpha_2}{2}.
	\end{equation}
	From \eqref{b3} we get, for all $t\in [0,T_0]$,
	\begin{equation*}
		\frac{d}{dt}\EE_2[u(t)\cutoff] \le C\eps^{-2}\sum_{i=1}^m\|u_i\|_{L^2(\Omega)}^2 + C\eps^{-(3+\xi)}.
	\end{equation*}
	Integrating on $(0,t)$, $t\in [0,T_0]$, it follows that
	\begin{equation*}
		\EE_2[u(t)\cutoff] \le \EE_2[u_0\cutoff] + C\eps^{-2}\sum_{i=1}^m\|u_i\|_{L^2(0,T;L^2(\Omega))}^2 + C\eps^{-(3+\xi)}T \le C(T)\eps^{-(3+\xi)}
	\end{equation*}
	where we apply Lemma \ref{lem:duality} and choose $\eps$ small enough. Now from \eqref{b2} and \eqref{norm_equivalent} we can estimate
	\begin{align*}
		\frac{\alpha_2}{2} = C\eps^{(2-\xi)\delta}\sum_{i=1}^m\|u_{i}(T_0)\cutoff\|_{L^2(\Omega)}^{2\xi} \le C\lambda_2^{-1}\eps^{(2-\xi)\delta}[\EE_2[u(T_0)\cutoff]]^{\xi} \le C(T)\eps^{(2-\xi)\delta - \xi(3+\xi)}.
	\end{align*}
	Because of \eqref{xi}, the exponent of $\eps$ on the right hand side is positive, and therefore, by choosing $\eps>0$ small enough, we get a contradiction. Thus, $T_0\ge T$ for such a small $\eps>0$, and we get from \eqref{b3} that 
	\begin{equation*}
		\EE_2[u(t)\cutoff] \le C(T,\eps) \quad \forall t\in (0,T).
	\end{equation*}
	With the partition of $\Omega$ as in Proposition \ref{pro:interp_Morrey} and the equivalent of norms \eqref{norm_equivalent} we obtain
	\begin{equation*}
		\sum_{i=1}^m\|u_i\|_{L^{\infty}(0,T;L^2(\Omega))} \le C(T)
	\end{equation*}
	which is enough to conclude the global existence of solution to \eqref{eq1.1}. The uniform-in-time boundedness follows similarly as in Theorem \ref{thm:mass_control}, so we omit it.
\end{proof}

\medskip
\noindent{\bf Acknowledgement.} B.Q. Tang received funding from the FWF project ``Quasi-steady-state approximation for PDE'', number I-5213. C. Sun is partially
supported by NSFC Grants No. 12271227. J. Yang are partially supported by NSFC Grants
No. 12271227 and China Scholarship Council (Contract No. 202206180025). A. Kostianko and S.Zelik are supported by the Russian Science Foundation (project 23-71-30008), Chapters 2 and 3. S. Zelik is also partially supported  by Laboratory of Dynamical Systems and Applications NRU HSE, grant of the Ministry of science and higher education of the RF, ag. \textnumero  075-15-2022-1101.

\medskip
A special thanks goes to Prof. M\'elanie Hall for pointing us out to the Belousov-Zhabotinsky reaction in subsection \ref{subsec:applications}.

\end{document}